\documentclass{amsart}
\usepackage{amssymb}
\usepackage[all]{xy}

\theoremstyle{plain}
\newtheorem{theorem}{Theorem}
\newtheorem{corollary}{Corollary}
\newtheorem{lemma}{Lemma}
\newtheorem{proposition}{Proposition}

\theoremstyle{definition}
\newtheorem{remark}{Remark}

\newcommand{\acts}{\curvearrowright}
\newcommand{\RR}{\mathbb{R}}
\newcommand{\ZZ}{\mathbb{Z}}
\newcommand{\SL}{\mathsf{SL}}
\newcommand{\PLFp}{\mathsf{PLF_+}}

\begin{document}

\title{Free subgroups acting properly discontinuously}

\author{Zoran \v{S}uni\'c}
\address{Department of Mathematics, Texas A\&M University, MS-3368, College Station, TX 77843-3368, USA}
\email{sunic@math.tamu.edu}
\thanks{This material is based upon work supported by the National Science Foundation under Grant No. DMS-1105520. }

\begin{abstract}
Given an action of a group $G$ on a topological space $X$, we establish a necessary and sufficient condition for the existence of a free subgroup $F$ of rank 2 of $G$ acting properly discontinuously on at least one nonempty, open, $F$-invariant subspace of $X$. In the case of a discrete topology (group action on a set $X$), the condition simply detects free subgroups of rank 2 acting freely on some orbit.
\end{abstract}



\dedicatory{Dedicated to the memory of Prof.~\'{G}.~\v{C}upona}

\maketitle

%------------------------------------------------------------------------
%------------------------------------------------------------------------

\section{Introduction}

We provide a criterion (Condition (ii) in Theorem~\ref{t:main}) for the existence of free subgroups acting properly discontinuously on an open subspace. If the criterion is satisfied one can write down explicitly generators of such a free subgroup (in terms of the elements satisfying the criterion).

\begin{theorem}\label{t:main}
Let $G \acts X$ be a left action of a group $G$ on a topological space $X$. The following conditions are equivalent.

(i) $G$ contains a free subgroup $F$ of rank 2 such that there exists a nonempty, open, $F$-invariant subspace $Z\subseteq X$ on which $F$ acts properly discontinuously (see Remark~\ref{r:properly}).

(ii) There exist nonempty, open subsets $A$ and $Y$ of $X$, with $A \subseteq Y \subseteq X$, and elements $a,a_0,a_1,a_2$ in $G$ such that
\begin{itemize}
 \item[] $Y = A \cup aA$,
 \item[] $a_0A$, $a_1A$, and $a_2A$ are disjoint,
 \item[] $Y$ is invariant under $a,a_0,a_1,a_2$ (see Remark~\ref{r:invariant}).
\end{itemize}

(iii) There exist nonempty, open subsets $A$ and $Y$ of $X$, with $A \subseteq Y \subseteq X$, such that, for every $n \geq 3$, there exist elements
$a,a_0,a_1,\dots,a_{n-1}$ in $G$ such that
\begin{itemize}
 \item[] $Y = A \cup aA$,
 \item[] $a_0A,a_1A,\dots,a_{n-1}A$ are disjoint,
 \item[] $Y$ is invariant under $a,a_0,a_1,\dots,a_{n-1}$.
\end{itemize}

(iv) there are elements $f_1$ and $f_2$ in $G$ and nonempty, open, disjoint subsets $U_0,U_1^+,U_1^-,U_2^+,U_2^-$ of $X$ such that
\begin{alignat*}{2}
 & f_1(U_0 \cup U_1^+ \cup U_2^+ \cup U_2^-) \subseteq U_1^+, \qquad && f_1^{-1}(U_0 \cup  U_1^- \cup U_2^+ \cup U_2^-) \subseteq U_1^- \\
 & f_2(U_0 \cup U_2^+ \cup U_1^+ \cup U_1^-) \subseteq U_2^+, \qquad && f_2^{-1}(U_0 \cup  U_2^- \cup U_1^+ \cup U_1^-) \subseteq U_2^-
\end{alignat*}

\vspace{5mm}

Moreover,
\begin{itemize}
\item[] if (ii) holds, $[a,a_1^{-1}a_0]$ and $[a,a_2^{-1}a_0]$ freely generate a copy of $F_2$.
\item[] if (iii) holds, with $n=4$, $a_0aa_1^{-1}$ and $a_2aa_3^{-1}$ freely generate a copy of $F_2$,
\item[] if (iii) holds, with $n \geq 2k+1$, for some $k \geq 2$, then $F=\langle a_1aa_2^{-1}, \dots,a_{2k-1}aa_{2k}^{-1} \rangle$ is a free group of rank $k$ acting properly discontinuously on the open, $F$-invariant, subspace $Fa_0A$.
\item[] if (iii) holds, with $n=5$, then (iv) holds with $f_1=a_1aa_2^{-1}$, $f_2=a_3aa_4^{-1}$, $U_0=a_0A$, $U_1^+=a_1A$, $U_1^-=a_2A$, $U_2^+=a_3A$, and $U_2^-=a_4A$.
\end{itemize}
\end{theorem}

\begin{remark}\label{r:onelesselement}
In Condition (ii) we may simply require that there exist elements $a,b,c$ in $G$ such that
\begin{itemize}
 \item[] $Y = A \cup aA$,
 \item[] $A$, $bA$, and $cA$ are disjoint,
 \item[] $Y$ is invariant under $a$, $b$, and $c$.
\end{itemize}
Indeed, if Condition (ii) holds then $A$, $a_0^{-1}a_1A$, and $a_0^{-1}a_2A$ are disjoint. In other words, we can always assume/choose $a_0=1$. The same remark is valid for Condition (iii). We keep $a_0$ in its general form both for conceptual and notational convenience.
\end{remark}

\begin{remark}\label{r:properly}
A properly discontinuous action of a group $H$ on a topological space $Z$ is an action $H \acts Z$ by homeomorphisms such that, for every point $z$ in $Z$, there exists an open neighborhood $U_z$ of $z$ such that $hU_z \cap U_z =\emptyset$, for all $h$ in $H$ such that $h \neq 1$ (see, for instance, Munkres~\cite{munkres:topology2nd}). Note that there are several nonequivalent definitions of properly discontinuous actions in the literature, which become equivalent only under various additional conditions (imposed on the groups or on the spaces on which they act).

The adopted definition implies that properly discontinuous actions are free. Moreover, in case the space $Z$ is discrete, properly discontinuous actions are precisely the free actions on the set $Z$.
\end{remark}

\begin{remark}\label{r:invariant}
We explicitly state that, here, invariance means equality, not just an inclusion. To be precise, by a subspace (set) $Y$ invariant under the action of an element $h$ (subgroup $H$) we mean a subspace (set) such that $hY=Y$ ($HY=Y$).
\end{remark}

\begin{remark}\label{r:schottky}
A group $F=\langle f_1,f_2\rangle$ satisfying Condition (iv) is called a \emph{quasi-Schottky} group by Margulis~\cite{margulis:ghys-circle}. Such a group acts properly discontinuously on the nonempty, open, $F$-invariant subspace $FU_0$ of $X$ (the implication (iv)$\implies$(i) is obvious). Groups satisfying some variation of this condition (which is itself a variation of the Klein Ping-Pong Lemma) provide a standard tool for finding free subgroups and have been used over and over again in almost all situations in which presence of free groups needs to be detected, from the classical work of Schottky, Klein, and their contemporaries, to the Tits Alternative for linear groups over fields of characteristic 0~\cite{tits:alternative}, to its more recent analogs for mapping class groups (Ivanov~\cite{ivanov:ta-mapping} and McCarthy~\cite{mccarthy:ta-mapping}, independently), and outer automorphism groups (Bestvina, Feighn, Handel~\cite{bestvina-f-h:ta-outer1,bestvina-f-h:ta-outer2}), and to the Ghys Alternative for groups acting on the circle (Margulis~\cite{margulis:ghys-circle}), to name just a few. Recall that Klein Ping-Pong Lemma, in one of its usual forms, states that if $G \acts X$ is a group action such that there exist two elements $a$ and $b$ in $G$ and two distinct subsets $A$ and $B$ of $X$ such that, for all nonzero $n$, $a^n(B) \subseteq A$ and $b^n(A) \subseteq B$, then $a$ and $b$ freely generate a free subgroup of $G$ of rank 2. One small difficulty with this condition is that it can only be used to verify that there is a free subgroup once we actually have generators for such a subgroup in hand (not to mention that we need to ``control'' all nonzero powers of $a$ and $b$). The flexibility Condition (ii) provides is that we do not need to know, or even suspect, in advance, any generators of a free subgroup; all we need is three elements $a$, $b$ and $c$ (as in Remark~\ref{r:onelesselement}) with fairly simple dynamics, and these three elements may even be (as we will see in examples) of finite order.
\end{remark}

By utilizing the left regular action (which is free) we obtain the following corollary.

\begin{corollary}\label{c:f2}
A group $G$ contains a free subgroup of rank 2 if and only if there exist a subset $A$ of $G$ and elements $a,b,c$ in $G$ such
that
\begin{itemize}
 \item[] $G = A \cup aA$, and
 \item[] $A$, $bA$, and $cA$ are disjoint.
\end{itemize}
\end{corollary}

%-------------------------------------------------------------
%----------------------------------------------------

\section{Examples}

\subsection{Free subgroups of $\PLFp(S^1)$}

It is known that the group $\PLFp(S^1)$ of piecewise linear, orientation preserving  homeomorphisms of the circle, with finitely many breaks, contains free subgroups of rank 2 (see, for instance, Section 4 in~\cite{brin-s:nofree}). We show that Condition (ii) may be used to reestablish this fact and to provide many explicit examples that are easy to construct.

We think of the circle as $S^1 =\RR/\ZZ=[0,1]/\{0=1\}$. Fix $n \geq 3$, let $A = [0,1/n)$,
\[
 a(t) =
 \begin{cases}
  (n-1)t + \frac{1}{n}, & 0 \leq t < \frac{1}{n} \\
  \frac{1}{n-1} t - \frac{1}{n(n-1)}, & \frac{1}{n} \leq t < 1 \\
 \end{cases},
\]
and, for $i=0,\dots,n-1$,
\[
 a_i = r^i,
\]
where $r(t) = t+\frac{1}{n}$ is the rotation by $1/n$. The rotated arcs $a_iA=r^iA=[i/n,(i+1)/n)$, for $i=0,\dots,n-1$, are disjoint.
Moreover, $a$ is a homeomorphism of order two, exchanging the intervals $A=[0,1/n)=a[1/n,1)$ and $aA=[1/n,1)$, after appropriate rescaling. Therefore $S^1=A \cup aA$, Condition (ii) is satisfied, and we may write down explicitly generators of a free group of rank 2 in terms of $a$ and $r$ (for instance $[a,r]$ and $[a,r^2]$, as suggested by the last part of Theorem~\ref{t:main}). One may readily construct many similar
examples of free subgroups in $\PLFp(S^1)$ (since it is so easy to compress/expand and rotate arcs around to either avoid each other or to cover
the entire circle, as we please).

If $k \geq 2$ and we set $n=2k+1$, a free group $F_k$ of rank $k$ is freely generated, according to Theorem~\ref{t:main}, by any set of $k$ generators of the form
\[
  a_{i_1}aa_{i_2}^{-1},~a_{i_3}aa_{i_4}^{-1},~\dots,~a_{i_{2k-1}}aa_{i_{2k}}^{-1},
\]
where $i_1,i_2,\dots,i_{2k}$ are distinct indices from $\{0,\dots,2k\}$. The free group $F$ obtained in this way acts freely on $Fa_jA$, where $j$ is the remaining index, the one that is not used in the construction of the generators. One straightforward choice is
\[
 s_i = a_i a a_{n-i}^{-1} = r^i a r^i, \qquad \text{ for } i=1,\dots,k
\]
The free group $F_k = \langle s_1,\dots,s_k\rangle$ acts freely on $F_kA$ (note that $F_kA$ has Lebesgue measure 1). In particular, it acts freely on the orbit of 0 and this action leads to a left order on the free group $F_k$ that extends the usual lexicographic order on the positive monoid $\{s_1,\dots,s_k\}^*$ based on $s_1<s_2<\dots<s_k$~\cite{sunic:free-lex}. Other choices of generators lead to other left orders on $F_k$.

%----------------------------------------------------

\subsection*{The examples of Bennett}

In his expository work, Pierre de la Harpe~\cite[page 30]{harpe:ggt} describes the noteworthy examples of free subgroups of $\PLFp(S^1)$, one for each finite rank $k \geq 2$, constructed by Bennett~\cite{bennett:free}. While the argument given by Bennett in his original article is clear and straightforward application of the classical Klein Ping-Pong Lemma, the generators are presented to the reader ``out of the blue''. We will show how one naturally arrives at the generators of the free subgroups of Bennet by making a few choices that differ slightly from the choices we made in the previous subsection.

Set, again, $n=2k+1$, and let $A=[0,1/(2k+1))$ and the homeomorphisms $a$ and $r$ (rotation by $1/(2k+1)$) be defined as before. Further, let
\begin{alignat*}{6}
 &b_0=1 ,     \qquad &&b_2= \ell, \qquad       &&b_4 = \ell^2,       \qquad &&\dots, \qquad &&b_{2k-2} = \ell^{k-1} \\
 &b_1=r^{-1}, \qquad &&b_3= \ell r^{-1}, \qquad &&b_5 = \ell^2 r^{-1}, \qquad &&\dots, \qquad &&b_{2k-1} = \ell^{k-1}r^{-1},
\end{alignat*}
where $\ell$ is the rotation by $1/k$. Each of the $2k$ homeomorphisms $b_0,\dots,b_{2k-1}$ is a rotation, and these $2k$ rotations translate $A$ to disjoint arcs.

For instance, for $k=3$, the disjoint translates of $A$ by $a_i=r^i$, $i=0,\dots,2k$, (as in the previous subsection) and by $b_j$, $j=0,\dots,2k-1$, (the choices we make in this subsection) are shown in Figure~\ref{f:bennet-translates} (notice the three gaps of size $1/21$ between $b_0A$ and $b_3A$, between $b_2A$ and $b_5A$, and between $b_4A$ and $b_1A$).
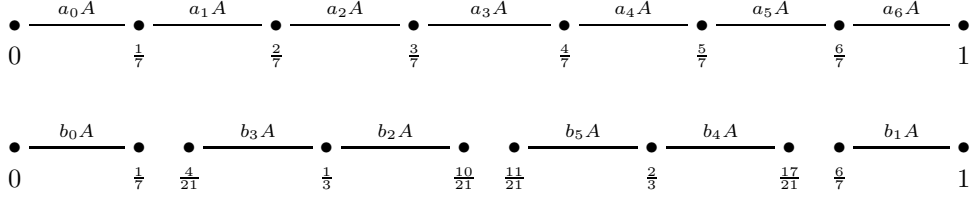
\begin{figure}[!ht]
\xymatrix@C=1pt@R=15pt{
 \bullet \ar@{-}[rrrrrr]^{a_0A} \ar@{}[d]|{\textstyle{0}} &&& &&&
 \bullet \ar@{-}[rrrrrr]^{a_1A} \ar@{}[d]|{\frac{1}{7}} &&& &&&
 \bullet \ar@{-}[rrrrrr]^{a_2A} \ar@{}[d]|{\frac{2}{7}} &&& &&&
 \bullet \ar@{-}[rrrrrr]^{a_3A} \ar@{}[d]|{\frac{3}{7}} &&& &&&
 \bullet \ar@{-}[rrrrrr]^{a_4A} \ar@{}[d]|{\frac{4}{7}} &&& &&&
 \bullet \ar@{-}[rrrrrr]^{a_5A} \ar@{}[d]|{\frac{5}{7}} &&& &&&
 \bullet \ar@{-}[rrrrrr]^{a_6A} \ar@{}[d]|{\frac{6}{7}} &&& &&&
 \bullet \ar@{}[d]|{\textstyle{1}}
 \\
 &&& &&& &&& &&& &&& &&& &&& &&& &&& &&& &&& &&& &&& &&&
 \\
 \bullet \ar@{-}[rrrrrr]^{b_0A} \ar@{}[d]|{\textstyle{0}} &&& &&&
 \bullet                        \ar@{}[d]|{\frac{1}{7}}        &&
 \bullet \ar@{-}[rrrrrr]^{b_3A} \ar@{}[d]|{\frac{4}{21}}  &&& &&&
 \bullet \ar@{-}[rrrrrr]^{b_2A} \ar@{}[d]|{\frac{1}{3}}   &&& &&&
 \bullet                        \ar@{}[d]|{\frac{10}{21}}      &&
 \bullet \ar@{-}[rrrrrr]^{b_5A} \ar@{}[d]|{\frac{11}{21}} &&& &&&
 \bullet \ar@{-}[rrrrrr]^{b_4A} \ar@{}[d]|{\frac{2}{3}}   &&& &&&
 \bullet                        \ar@{}[d]|{\frac{17}{21}}      &&
 \bullet \ar@{-}[rrrrrr]^{b_1A} \ar@{}[d]|{\frac{6}{7}}   &&& &&&
 \bullet \ar@{}[d]|{\textstyle{1}}
 \\
 &&& &&& &&& &&& &&& &&& &&& &&& &&& &&& &&& &&& &&& &&&
}
\caption{Translates of $A$ by $a_i$, $i=0,\dots,6$ and $b_j$, $j=0,\dots,5$}
\label{f:bennet-translates}
\end{figure}
Moreover, the $2k$ translated arcs of length $1/(2k+1)$ leave enough room (in fact, there are $k$ gaps of size $\frac{1}{k(2k+1)}$) for another translate of $A$ to be squeezed in, say, between $b_0A = \left[0,\frac{1}{2k+1}\right)$ and $b_3A=\left[\frac{k+1}{k(2k+1)},\frac{1}{k}\right)$ (this translate cannot be by a rotation, since $A$ needs to be compressed by a factor of at least $k$). Therefore, by Theorem~\ref{t:main},
\[
 t_i = b_{2i} a b_{2i+1}^{-1} =  \ell^i a r \ell^{-i}, \qquad \text{ for } i=0,\dots,k-1,
\]
freely generate a subgroup of rank $k$ of $\PLFp(S^1)$ acting freely on the orbit of any point between $\frac{1}{2k+1}$ and $\frac{k+1}{k(2k+1)}$ (Bennet singles out the point $1/(2k)$ in his article).

For completeness, we provide explicit formulas for the homeomorphisms $s_i=r^iar^i$, $i=1,\dots,k$, from the previous subsection, and the homeomorphisms $t_j=\ell^i ar\ell^{-1}$, $j=0,\dots,k-1$, defined by Bennet. Denote $s_0=t_0=ar$, Direct calculation gives
\[
 s_0(x) = t_0(x) =
  \begin{cases}
    \frac{1}{2k} x, & 0 \leq x < \frac{2k}{2k+1} \\
    2k x - (2k-1), & \frac{2k}{2k+1} \leq x < 1.
  \end{cases}
\]
Therefore, for $i=1,\dots,k$,
\[
 s_i = s_0 \left(x+\frac{i-1}{2k+1}\right) + \frac{i}{2k+1}.
\]
and, for $j=0,\dots,k-1$,
\[
 t_j = t_0 \left(x-\frac{j}{k}\right) + \frac{j}{k}.
\]

\begin{remark}
To be pedantic, Bennett does not define his examples as groups of homeomorphisms of the circle, but rather as groups of homeomorphisms of the line. However, the groups $\langle t_0,\dots,t_{k-1}\rangle$, for $k \geq 2$, are just projections (along the projection from $\RR$ to $S^1=\RR/\ZZ$) of the groups he originally constructed.
\end{remark}
%----------------------------------------------------

\subsection{Free subgroups of $\SL_2(\RR)$ and $\SL_2(\ZZ)$} A straightforward application of Theorem~\ref{t:main} to $\SL_2(\RR)$ quickly rediscovers the ``classics''. We see, once again, that generators for many free subgroups may be easily constructed out of elements satisfying the requirements of Condition (ii) or Condition (iii) and that such elements are often easy to find.

Consider the standard (linear) action of $\SL_2(\RR)$ on the plane $X=\RR^2$. Let $Y= X \setminus \{(0,0)\}$ and
\[
 A = \{ (x,y) \in Y \mid  x > 0,~ y \geq 0\} \cup \{ (x,y) \mid  x < 0,~y \leq 0 \}.
\]
In other words, $A$ is the union of the first and the third quadrant, including the $x$-axis, but excluding the $y$-axis (and the origin). Let
\[
 a = \begin{pmatrix} 0 & -1 \\ 1 & 0 \end{pmatrix}, ~
 a_0 = \begin{pmatrix} 1 & \alpha \\ 0 & 1 \end{pmatrix},~
 a_1 = \begin{pmatrix} 1 & 0 \\ \beta & 1 \end{pmatrix}, ~
 a_2 = \begin{pmatrix} 0 & -1 \\ 1 & \gamma \end{pmatrix}, ~
 a_3 = \begin{pmatrix} -\delta & -1 \\ 1 & 0 \end{pmatrix},
\]
where $\alpha, \beta, \gamma,\delta >0$, $\alpha\beta \geq 1$, and $\gamma\delta \geq 1$. It is easy to check that $Y=A \cup aA$ (in fact, since $a$
is the rotation by $\pi/2$, $aA=Y \setminus A$), that $a_0A=-a_0A$, $a_1A=-a_1A$, $a_2A=-a_2A$, and $a_3A=-a_3A$, and that these four sets are disjoint (this is where the conditions $\alpha\beta \geq 1$, and $\gamma\delta \geq 1$ play a role). Therefore, each of the following pairs of matrices freely generate a copy of $F_2$ in $\SL_2(\RR)$:
\begin{alignat*}{5}
 -& a_2aa_0^{-1} &&= \begin{pmatrix} 1 & -\alpha \\ -\gamma & 1+\alpha \gamma \end{pmatrix} \qquad &&\text{and} \qquad
 -&&a_3aa_1^{-1} &&= \begin{pmatrix} 1+\beta\delta & -\delta  \\ -\beta & 1 \end{pmatrix}, \\
  & a_1aa_0^{-1} &&= \begin{pmatrix} 0 & -1 \\ 1 & -(\alpha+\beta) \end{pmatrix} \qquad &&\text{and} \qquad
  &&a_3aa_2^{-1} &&= \begin{pmatrix} 0 & -1 \\ 1 & \gamma+\delta \end{pmatrix}, \\
 -& a_3aa_0^{-1} &&= \begin{pmatrix} 1 & \alpha+\delta \\ 0 & 1 \end{pmatrix} \qquad &&\text{and} \qquad
 -&&a_2aa_1^{-1} &&= \begin{pmatrix} 1 & 0  \\ \beta+\gamma & 1 \end{pmatrix}. \\
\end{alignat*}
By choosing $\alpha$, $\beta$, $\gamma$ and $\delta$ in $\ZZ$, we obtain free subgroups in $\SL_2(\ZZ)$. If either $\alpha\beta>1$ or $\delta\gamma>1$, there is enough space left in the plane for a fifth translate $a_4A$, and in that case we may easily point out free orbits. In fact, the constructed free group $F$ acts freely on $F(Y \setminus (a_0A \cup a_1A \cup a_2A \cup a_3A))$. 

Note that, since $\alpha, \beta, \gamma,\delta >0$, $\alpha\beta \geq 1$, and $\gamma\delta \geq 1$, the second and the third pair may be rewritten as
\begin{alignat*}{3}
  &\begin{pmatrix} 0 & -1 \\ 1 & -s \end{pmatrix} \qquad &&\text{and} \qquad
 &&\begin{pmatrix} 0 & -1 \\ 1 & t \end{pmatrix}, \\
  &\begin{pmatrix} 1 & u \\ 0 & 1 \end{pmatrix} \qquad &&\text{and} \qquad
 &&\begin{pmatrix} 1 & 0  \\ v & 1 \end{pmatrix}. \\
\end{alignat*}
where $s,t \geq 2$, $u,v >0$, and $uv \geq 4$. The construction is already rich enough to reestablish many classical results about free subgroups of $\SL_2(\RR)$, including the following.

\begin{proposition}
For any two non-elliptic elements $f$ and $g$ (including the possibility that $f=g$) in $\SL_2(\RR)$, there exists a conjugate $g^h$ of $g$ in $\SL_2(\RR)$ such that $\langle f, g^h \rangle$ is free of rank 2.
\end{proposition}

%-------------------------------------------------------------
%-------------------------------------------------------------

\section{Proof of Theorem~\ref{t:main}}

\subsection{Condition (i) implies Condition (ii)}

\begin{proof}[Proof of Theorem~\ref{t:main}, (i) implies (ii)]
Let $F=\langle a,b \rangle$ be freely generated by $a$ and $b$, and $A'$ be the subset of $F$ consisting of the elements of $F$
that, written in freely reduced form in terms of $a$ and $b$, start either with $a$ or with $a^{-1}$ (note that the empty word is not in $A'$). Then
\begin{itemize}
 \item[] $F = A' \cup aA'$, and
 \item[] $A'$, $bA'$, and $b^2A'$ are disjoint.
\end{itemize}

Choose a point $z$ in $Z$. Since the action of $F$ is properly discontinuous, we may choose an open neighborhood $U$ of $z$ such that $(F\setminus 1)U \cap U = \emptyset$. Note that in case of a discrete space $X$, the set $U$ may be chosen to be the singleton $\{z\}$, or any set of points $U$, $U \subseteq Z$, such that $z \in U$ and that no two points in $U$ come from the same $F$-orbit (in other words, $U$ is a partial system of orbit representatives). The set $Y=FU$ is nonempty, open, $F$-invariant subset of $Z$.

Let
\[
 A = A'U.
\]
The set $A$ is a nonempty, open subset of $Y$ and
\[
 A \cup aA = A'U \cup aA'U = (A' \cup aA')U = FU = Y.
\]

We claim that, for any $h_1,h_2$ in $F$ such that $h_1A' \cap h_2A' = \emptyset$,
\[
 h_1A \cap h_2A = \emptyset.
\]
Indeed,
\begin{align*}
 h_1A \cap h_2A &= h_1A'U \cap h_2A'U =
 \bigcup_{u,v \in U} (h_1A'u \cap h_2A'v) = \\
 &\stackrel{(*)}{=} \bigcup_{u \in U} (h_1A'u \cap h_2A'u)\stackrel{(**)}{=} \bigcup_{u \in U} (h_1A' \cap h_2A')u = \emptyset,
\end{align*}
where equality~(*) is due to the fact that $h_1A'$ and $h_2A'$ are subsets of $F$ (which implies that $h_1A'u \cap h_2A'v \subseteq Fu \cap
Fv = \emptyset$, when $u \neq v$), and equality~(**) is due to the fact that the action of $F$ on $Z$ is free (recall Remark~\ref{r:properly}).

Therefore, since $A'$, $bA'$, and $b^2A'$ are disjoint subsets of $F$, the open sets $A$, $bA$ and $b^2A$ are disjoint subsets of $Y$ and Condition (ii) is satisfied for $a_0=1$, $a_1=b$, $a_2=b^2$.
\end{proof}

%----------------------------------------------------------

\subsection{Condition (ii) implies Condition (iii)}\label{ss:23}

We first need the following Inequality Lemma.

\begin{lemma}[Inequality Lemma]
Let $G \acts X$ be a left group action, $A$ and $Y$ be subsets of $X$, with $A \subseteq Y \subseteq X$, and, for some $n \geq 3$, let
$a,a_0,\dots,a_{n-1}$ be elements of $G$ such that
\begin{itemize}
 \item[] $Y = A \cup a A$,
 \item[] $a_0A,a_1A,\dots,a_{n-1}A$ are disjoint,
 \item[] $Y$ is invariant under $a,a_0,\dots,a_{n-1}$.
\end{itemize}
Then, for all $i,j,\ell$ in $\{0,\dots,n-1\}$, if $i \neq j$,
\[
  a_\ell a^{-1}a_j^{-1}(a_iA) \subseteq a_\ell A \qquad \text{and} \qquad
  a_\ell aa_j^{-1}(a_iA) \subseteq a_\ell A.
\]
\end{lemma}

\begin{proof}
Multiplying the equality $Y = A \cup aA$ by $a_j$ on the left yields
\[
 Y = a_jA \cup a_jaA,
\]
and since $a_iA$ is disjoint from $a_jA$ it follows that $a_iA \subseteq a_jaA$ or, equivalently,
\[ a_\ell a^{-1}a_j^{-1}(a_iA) \subseteq a_\ell A. \]

Since $Y=A \cup aA$ is equivalent to $Y=A \cup a^{-1}A$ we see that, by symmetry, the inequality obtained by replacing $a^{-1}$ by $a$ is also valid.
\end{proof}

\begin{proof}[Proof of Theorem~\ref{t:main}, (ii) implies (iii)]
Assume Condition (ii) holds.

It is sufficient to show that if the statement that there are $n$ elements $a_0,a_1,\dots,a_{n-1}$ such that
\begin{itemize}
 \item[] $a_0A,a_1A,\dots,a_{n-1}A$ are disjoint, and
 \item[] $Y$ is invariant under $a_0,\dots,a_{n-1}$,
\end{itemize}
is valid for some $n \geq 3$, then it is valid for $n+1$ as well.

Assume that the statement is valid for some $n \geq 3$. The Inequality Lemma implies that
\[
 a_{n-1}aa_{n-1}^{-1}(a_0A) \subseteq a_{n-1}A \qquad \text{and} \qquad a_{n-1}aa_{n-1}^{-1}(a_1A) \subseteq a_{n-1}A.
\]
Since $a_0A$ and $a_1A$ are disjoint, so are $a_{n-1}aa_{n-1}^{-1}(a_0A)$ and $a_{n-1}aa_{n-1}^{-1}(a_1A)$, and since they are both in $a_{n-1}A$, they are also disjoint from $a_0A,a_1A,\dots,a_{n-2}A$. Therefore
\[ a_0A,~a_1A,~\dots,~a_{n-2}A,~a_{n-1}aa_{n-1}^{-1}a_0A,~a_{n-1}aa_{n-1}^{-1}a_1A
\]
are $n+1$ disjoint translates of $A$.
\end{proof}

%----------------------------------------------------------------

\subsection{Condition (iii) implies Condition (iv)}\label{ss:34}

\begin{proof}[Proof of Theorem~\ref{t:main}, (iii) implies (iv)]
Let $A$ and $Y$ be subsets of $X$, with $A \subseteq Y \subseteq X$, such that, there exist elements
$a,a_0,a_1,a_2,a_3,a_4$ in $G$ such that
\begin{itemize}
 \item[] $Y = A \cup aA$,
 \item[] $a_0A,a_1A,a_2A,a_3A,a_4A$ are disjoint,
 \item[] $Y$ is invariant under $a,a_0,a_1,a_2,a_3,a_4$.
\end{itemize}

It follows from the Inequality Lemma that
\begin{alignat*}{2}
 & a_1aa_2^{-1}(a_0A \cup a_1A \cup a_3A \cup a_4A) \subseteq a_1A, \qquad &&(a_1aa_2^{-1})^{-1}(a_0A \cup a_2A \cup a_3A \cup a_4A) \subseteq a_2A \\
 & a_3aa_4^{-1}(a_0A \cup a_3A \cup a_1A \cup a_2A) \subseteq a_3A, \qquad &&(a_3aa_4^{-1})^{-1}(a_0A \cup a_4A \cup a_1A \cup a_2A) \subseteq a_4A
\end{alignat*}
Therefore, Condition (iv) is satisfied for $f_1=a_1aa_2^{-1}$, $f_2=a_3aa_4^{-1}$, $U_1^+ = a_1A$, $U_1^- = a_2A$, $U_2^+ = a_3A$, $U_2^- = a_4A$,
and $U_0 = a_0A$.
\end{proof}

%-----------------------------------------------------------------------------

\subsection{Wrap up of the Proof of Theorem~\ref{t:main}}

Of course, the implication (iv)$\implies$(i) is well known. Moreover, an analogous condition with $k$ elements $f_1,\dots,f_k$ and $2k+1$ sets $U_0,U_1^\pm,\dots,U_k^\pm$ with appropriate attracting properties ensures that $F=\langle f_1,\dots,f_k \rangle$ is a free subgroup of rank $k$ acting properly discontinuously on $FU_0$.

It remains to track down explicit generators for a free group under Condition (ii) and Condition (iii).

If Condition (iii) is satisfied, with $n=5$, then by the proof of (iii)$\implies$(iv) in Subsection~\ref{ss:34}, $a_1aa_2^{-1}$ and $a_3aa_4^{-1}$ freely generate a quasi-Schottky free subgroup $F$ of rank 2 acting properly discontinuously on $Fa_0A$. More generally, if Condition (iii) is satisfied with $n \geq 2k+1$, then $F=\langle a_1aa_2^{-1},\dots,a_{2k-1}aa_{2k}^{-1}\rangle$ is a free subgroup of rank 2 acting properly discontinuously on $Fa_0A$.

If Condition (iii) is satisfied, with $n=4$, then by the Inequality Lemma and Klein Ping-Pong Lemma with table halves $a_2A \cup a_3A$ and $a_0A \cup a_1A$, the elements $a_0aa_1^{-1}$ and $a_2aa_3^{-1}$ freely generate a copy of $F_2$.

If Condition (ii) is satisfied, then, by the proof of (ii)$\implies$(iii) in Subsection~\ref{ss:23}, Condition (iii) is satisfied for $n=4$, and the roles of $a_0$, $a_1$, $a_2$ and $a_3$ in Condition (iii) may be played by $a_1$, $a_2$, $a_0a^{-1}a_0^{-1}a_1$, and $a_0a^{-1}a_0^{-1}a_2$. Then, by using the explicit generators already established when Condition (iii) is satisfied (see the previous paragraph), $a_0a^{-1}a_0^{-1}a_1aa_1^{-1}$ and $a_0a^{-1}a_0^{-1}a_2aa_2^{-1}$ freely generate a copy of $F_2$. After conjugation by $a_0$, we see that $[a,a_1^{-1}a_0]$ and $[a,a_2^{-1}a_0]$ also generate a copy of $F_2$.

\subsection{Proof of Corollary~\ref{c:f2}}

\begin{proof}[Proof of Corollary~\ref{c:f2}]
Assume that $G$ contains a copy $F$ of the free group of rank 2, and let $F=\langle a,b\rangle$. Consider $G$ as a discrete space and the left regular action $F \acts G$ by left multiplication. This action is free. At this point we are in the setup of the proof of Theorem~\ref{t:main}, (i)$\implies$(ii), with $G$ in the role of $Z$. If we let any right transversal $T$ of $F$ in $G$ play the role of the set $U$ (recall, from the proof of Theorem~\ref{t:main}, (i)$\implies$(ii), that any partial system of $F$-orbit representatives may play this role), the proof shows that the role of $Y$ is played by $FU=FT=G$. Therefore $G = A \cup aA$, and $A$, $bA$, and $cA$ are disjoint, where $A=A'T$, and $A'$ is the set of nontrivial elements in the free group $F$ that, in their reduced form, start in $a$ or in $a^{-1}$.

The other direction follows directly from Theorem~\ref{t:main}, (ii)$\implies$(i), by considering the left regular action of $G$ on itself.
\end{proof}

%---------------------------------------------------------------
%---------------------------------------------------------------

\subsection*{Acknowledgments}

I would like to thank Pierre de la Harpe, for sharing his thoughts and enthusiasm for the subject through our conversations and correspondence, and the referee, for suggesting several improvements.

%---------------------------------------------------------------
%---------------------------------------------------------------

\def\cprime{$'$}

%\bibliographystyle{alpha}
%\bibliography{../smath}

\end{document}